\documentclass[a4paper,12pt,reqno]{amsart}
\pdfoutput=1

\usepackage{hyperref}
\usepackage{amssymb,mathrsfs,mathdots,latexsym}
\usepackage{amsfonts,amsthm,amsmath}
\usepackage{microtype}

\theoremstyle{plain}

\newtheorem{theorem}{Theorem}[section]
\newtheorem*{theorem*}{Theorem}
\newtheorem{corollary}[theorem]{Corollary}

\newtheorem{proposition}[theorem]{Proposition}

\numberwithin{equation}{section}

\theoremstyle{remark}

\newtheorem{definition}[theorem]{Definition}

\def \N{\mathbb N} 
\def\bi{{\bf i}}    \def\bj{{\bf j}}   \def\bk{{\bf k}}

\def\wtilde{\widetilde}

\def \N{\mathbb N}

\newcommand{\coloneqq}{\mathrel{\mathop:}=}

\title[Application of quasideterminants]%
{Application of quasideterminants to the inverse of block triangular matrices over noncommutative rings}
\date{\today}


\author[X. Zhan]{Xuzhou Zhan}
\address[X. Zhan]{School of Mathematical Sciences, South China Normal University,
    Guangzhou 510631, China}
\email{xzzhan@m.scnu.edu.cn}

\begin{document}

\begin{abstract}
Given a block triangular matrix $M$ over a noncommutative ring with invertible diagonal blocks, this work gives two new   representations
 of its inverse $M^{-1}$. Each block element of $M^{-1}$ is explicitly expressed  via a quasideterminant of a  submatrix of $M$ with the  block Hessenberg type.
Accordingly another representation for each inverse block is attained, which is in terms of recurrence relationship with multiple terms among blocks of $M^{-1}$.
The latter result allows us to perform an off-diagonal rectangular perturbation analysis for the inverse calculation of $M$. An example is given to illustrate the effectiveness of our results.

\end{abstract}

\maketitle

\section{Introduction}

As is pointed out in  \cite{GGRW},  quasideterminants play an important role in noncommutative algebra as  determinants do in commutative algebra.
These extended noncommutative determinants,
which are initiated in a series of literature by  Gelfand et al. \cite{GR,GGRW,GelRF}, have various applications in branches of noncommutative mathematics such as  noncommutative symmetric functions \cite{GKLLRT}, noncommutative integrable systems \cite{RS, EGR,EGR1}, quantum algebras  and Yangians \cite{GR,GR1,GR2,KL,Mol,Mol1,MolR}.

This work pays attention to a typical type of quasideterminants of block Hessenberg matrices  over noncommutative rings, which are hereinafter called Hessenberg quasideterminants. Compared with general quasideterminants, these special quasideterminants have explicit expressions  and, moreover, allow us to dig the value of quasideterminants in the matrix inverse problems.

For  the study of the inverse problems of matrices or typically triangular matrices with noncommutative entries  the reader is referred to \cite{CRV}, \cite{GGRW},  \cite{Mary},  \cite{MuCr}, \cite{CY}, \cite{LHL}, etc. Suppose that ${M}$ is a $m\times m$ block triangular matrix   over a noncommutative ring $ \mathcal R$ with the $n$-by-$n$ block decomposition
 \begin{equation}\label{UpperMatrix}
 {M}=\begin{bmatrix}
 M_{11} & \cdots & M_{1n}\\
  & \ddots & \vdots\\
  & & M_{nn}
 \end{bmatrix},
 \end{equation}
where each element $M_{jk}$ is a ${s_j\times s_k}$ matrix  ($m=\sum_{j=1}^{n}s_j$). 
For the simple case when $n=2$, if $M_{11}$ and $M_{22}$ are invertible in the rings of $s_1\times s_1$ matrices and $s_2\times s_2$ matrices over $\mathcal R$  respectively, then $M$ is invertible in the ring of $m\times m$ matrices over $\mathcal R$ and 
$
{M}^{-1}=\begin{bmatrix}
 M^{-1}_{11} &  -M_{11}^{-1}M_{12}M^{-1}_{22}\\
   & M^{-1}_{22}
 \end{bmatrix}.
$ 
 For the general case, decompose $M$ into the  $2$-by-$2$ block form
$$
M=\left[\begin{array}{c|c}
M_{11} & M_{(12)} \\ \hline
0 & M_{(22)}
\end{array}\right].
$$
 The invertibility of $M$ can be fulfilled through that of $M_{11}$ and another block  matrix $M_{(22)}$ of smaller size. 
 Via iteration on $M_{(22)}$ and its block submatrices, it is clear that all  block triangular matrices $M$ with invertible diagonnal element $M_{kk}$ are invertible as well 
  (In fact, for the particular case when each $s_k$ equals $1$ and $\mathcal R$ is a Dedekind-finite ring, these matrices exactly  make up of the unit group of the ring of all  triangular matrices over $\mathcal R$) and, in this case, the inverse $M^{-1}$ can be iteratively calculated. 
 
 This iterative algorithm is a direct and convenient way to calculate the inverse $M^{-1}$ in some sense, whereas it does not provide an explicit and simple expression to each single block or element of   $M^{-1}$. To give an improvement,
this work is devoted to  use Hessenberg quasideterminants as a new representing tool.  
Some preliminaries for these special quasideterminants are given in 
 the following auxiliary  section. By decomposing $M^{-1}$ into the analogous $n$-by-$n$ block form as $M$, we  express 
each block explicitly  in Section \ref{Sec:Inverse} via a Hessenberg quasideterminant of a  submatrix of $M$ with  block Hessenberg type.
Accordingly,  through the recurrence relation with multiple terms among the $n^2$ blocks of $M^{-1}$, another elegant expression for
each single block can be fulfilled.
As a consequence of the latter representation, the corresponding behaviour  is described for each single block of  $M^{-1}$   when an off-diagonal rectangular pertubation is imposed on $M$.  An example is given in Section \ref{Sec:Example}  to illustrate the effectiveness of our results. 

\section{Hessenberg quasideterminants}
Throughout the paper, let $\N$ denote the set of all positive integers. Generally we may always assume that $m,n\in \N$
and $\mathcal R$ is a noncommutative ring with the additive identity $0_{\mathcal R}$  and the multiplicative identity $I_{\mathcal R}$. Let~$\mathcal R^{m\times n}$ stand for the set
of all $m\times n$ matrices over $\mathcal R$.
Suppose that $M:=[M_{jk}]_{j,k=1}^n\in \mathcal R^{m\times m}$, where each block element $M_{jk}\in \mathcal R^{s_j\times s_k}$ ($m=\sum_{j=1}^{n}s_j$). For $j,k=1,\ldots,n$, we use ${\rm row}_j ({M})$ and ${\rm col}_k ({M})$ to signify the $j$-th block row vector and the $k$-th block column vector of ${M}$, i.e.
\begin{align*}
&{\rm row}_j ({M}):=(M_{j1},\ldots, M_{jn}),\\
&{\rm col}_k ({M}):=\begin{bmatrix}
M_{1k}\\
 \vdots\\
 M_{nk}
\end{bmatrix}.
\end{align*}
We denote by ${M}_{(j;k)}$ the $(n-1)\times (n-1)$ block submatrix of ${M}$ that removed by ${\rm row}_j ({M})$ and ${\rm col}_k ({M})$.

We begin with an introduction of quasideterminants.  
\begin{definition}
Let ${M}$ be an  $m\times m$ matrix over $\mathcal R$  with a $2$-by-$2$ block decomposition
$$
M=:\left[\begin{array}{c|c}
M_{(11)} & M_{(12)} \\ \hline
M_{(21)} & M_{(22)}
\end{array}\right].
$$
If $M_{(21)}$ is invertible, then the {\it quasideterminant} of $M$ {\it with block index} $(1,2)$, is the following expression
\begin{equation*}
\left| M \right|_{(12)}=M_{(12)}-M_{(11)}M_{(21)}^{-1} M_{(22)}.
\end{equation*}
\end{definition}
Motivated by the idea in \cite[Definition 3.1, P. 12]{Olv},  our definition is slightly  different from the Gelfand-Retakh formulation (see e.g.  \cite[Section 1]{GR}). Compared with the latter, 
all blocks  $M_{(jk)}$
of the matrix $M$ do not necessarily have the same size which belong to the same noncommutative ring, and hence the multiplication of blocks can only be allowed if their sizes are ``compatible''.

The next typical  quasideterminant is indebted to the above definition, which is  defined for a $m\times m$  block upper Hessenberg matrix ${H}$ over $\mathcal R$ with the $n$-by-$n$ block decomposition
\begin{equation}\label{HenssenbergMatrix}
{H}:=\left[\begin{array}{ccc|c}
M_{01} &  \cdots &  M_{0,n-1} & M_{0n}\\ \hline
M_{11} &  \cdots & M_{1,n-1} & M_{1n}\\
 & \ddots & \vdots &  \vdots\\
 &   & M_{n-1,n-1}
 & M_{n-1,n}
\end{array}\right],
\end{equation}
where each element $M_{jk}\in \mathcal R^{s_j\times s_k}$ ($m=\sum_{j=0}^{n-1} s_j=\sum_{j=1}^{n} s_j$). In order for the quasideterminant of ${H}$ to be well-defined, we additionally require that each $M_{jj}$ is invertible.

\begin{definition}\label{DefQuasi}
Let ${H}$ be an $m\times m$  block upper Hessenberg matrix  with the block decomposition \eqref{HenssenbergMatrix}, where each block $M_{jk}\in \mathcal R^{s_j\times s_k}$ and each $M_{jj}$ is invertible ($m=\sum_{j=0}^{n-1} s_j=\sum_{j=1}^{n} s_j$).
Then the \emph{Hessenberg quasideterminant of} $H$, denoted by $\left| H \right|$, is equal to 
\begin{equation*}
\left| H \right|=M_{0n}- [M_{01}, \cdots, M_{0,n-1} ]  \begin{bmatrix}
 M_{11} & \cdots & M_{1,n-1}\\
  & \ddots  & \vdots\\
 & & M_{n-1,n-1}
 \end{bmatrix}^{-1}
\begin{bmatrix}
M_{1n}\\
\vdots\\
M_{n-1,n}
\end{bmatrix},
\end{equation*}

\end{definition}

Given  a block Hessenberg  matrix ${H}$ as in \eqref{HenssenbergMatrix}, one can associate another block Hessenberg matrix~${H}_{\diamond}$ as
  \begin{equation*}
{H}_{\diamond}:=\mbox{diag} (I_{\mathcal R^{s_0\times s_0}}, M^{-1}_{11}, M^{-1}_{22},\cdots, M^{-1}_{n-1,n-1})\cdot {H}.
\end{equation*}
Analogous to \cite[Subsection 1.2, p. 93]{GR}, we have $\left|H\right|=\left|{H}_{\diamond}\right|$.
It follows from \cite[Proposition 1.2.9]{GGRW} that
\begin{equation}\label{QuasiRepreElement}
\left|H\right|=\left|{H}_{\diamond}\right|=M_{0n}+
\sum_{1\leq j_1<\cdots <j_k <n}(-1)^{k}  M_{0j_1} M_{j_1,j_1}^{-1}  M_{j_1,j_2}\cdots M_{j_k,j_k}^{-1} M_{j_k,n}.
\end{equation}
Accordingly, $\left|H\right|$ can be factorized  into two ``smaller size" Hessenberg quasideterminants as
\begin{align}
&\left|H \right|=M_{0n}+ \sum_{1< j_1<\cdots <j_k <n}(-1)^{k}  M_{0j_1}   M_{j_1,j_2}\cdots M_{j_k,n}
 \nonumber \\
&- M_{01}\cdot  \sum_{1< j_2<\cdots <j_k <n}(-1)^{k-1}  M_{1,j_2}\cdots  M_{j_k,n} \nonumber \\
&=\left|H_{(2;1)}\right|-M_{01}\cdot \left|{H}_{(1;1)}\right|. \label{HH1112}
\end{align}

The following result is a more general factorization for  Hessenberg quasideterminants than \eqref{HH1112}. Suppose that $M:=[M_{jk}]_{j,k=1}^n\in \mathcal R^{m\times m}$, where each block element $M_{jk}\in \mathcal R^{s_j\times s_k}$ ($m=\sum_{j=1}^{n}s_j$).
For $v,w,j,k=1,\ldots,n$ such that $1\leq v+j-1, w+k-1\leq n$, ${M}^{(v,w)}_{[j,k]}$ represents such a block submatrix of ${M}$ as
$$
{M}^{(v,w)}_{[j,k]}:=\begin{bmatrix}
M_{v,w} & \cdots &  M_{v,w+k-1}\\
\vdots & \ddots & \vdots\\
M_{v+j-1,w} & \cdots & M_{v+j-1,w+k-1}\end{bmatrix}.
$$
For simplicity we write ${M}^{(v,w)}_{[k,k]}$ for ${M}^{(v,w)}_{[k]}$ and further ${M}_{[k]}$ for ${M}^{(1,1)}_{[k]}$. For some technical reasons, we also add that ${M}^{(v,w)}_{[0]}=I_{\mathcal R}$.

\begin{proposition}\label{FactorHessenberg} Let ${H}$ be an $m\times m$  block upper Hessenberg matrix with the $n$-by-$n$ block decomposition \eqref{HenssenbergMatrix}, where each element $M_{jk}\in \mathcal R^{s_j\times s_k}$ $(m=\sum_{j=0}^{n-1} s_j=\sum_{j=1}^{n} s_j)$. If $M_{kk}$ is invertible in the ring $\mathcal R^{s_k\times s_k}$, then   for any $j=1,\ldots,n$, $\left| H\right|$ can be factorized as
\begin{equation}\label{SmallHessenberg}
\left| H\right|=\sum_{l=0}^{j-1}\sum_{k=j}^n \sideset{_{l}}{}{\mathop{H}}\cdot M_{lk}\cdot H_k,
\end{equation}
where
$
\sideset{_{l}}{}{\mathop{H}}:=\begin{cases}
-I_{\mathcal R^{s_0\times s_0}},& l=0,\\
\left|{H}_{[l]}\right| \cdot M_{ll}^{-1}, & \mbox{others},
\end{cases}
$ and
$
H_k:=\begin{cases}
-I_{\mathcal R^{s_n\times s_n}}, & k=n,\\
M_{kk}^{-1} \cdot \left|{H}_{[n-k]}^{(k+1)}\right|, & \mbox{others,}\\
\end{cases}
$ or equivalently, in the  matrix form
\begin{align*}
\left|H\right|=&\left(-I_{\mathcal R^{s_0\times s_0}}, \left|{H}_{[1]}\right|,\cdots, \left|{H}_{[j-1]}\right|\right)
\begin{bmatrix}
I_{\mathcal R^{s_0\times s_0}} & & & \\
& M_{11} & & \\
& & \ddots & \\
& & & M_{j-1,j-1}
\end{bmatrix}^{-1}\\
&\cdot  {H}^{(1,j)}_{[j,n-j+1]} \cdot \begin{bmatrix}
 M_{jj} & & &\\
 & \ddots & &\\
 & & M_{n-1,n-1} &\\
 & & & & I_{\mathcal R^{s_n\times s_n}}\\
\end{bmatrix}^{-1}
\begin{bmatrix}
 \left|{H}^{(j+1)}_{[n-j]}\right|\\
 \vdots\\
 \left|{H}^{(n)}_{[1]}\right|\\
-I_{\mathcal R^{s_n\times s_n}}
\end{bmatrix}.
\end{align*}
\end{proposition}

\begin{proof}
Our proof will be given via induction.
Case I: $M_{kk}=I_{\mathcal R}$ for all $k=1,\ldots, n$. \eqref{SmallHessenberg} is obviously  true  for $n=2$.
Suppose that this formula holds for $n\leq k$ as well. Now consider the case that $n=k+1$.
On grounds of the induction, one obtains that
\begin{multline}
\left|{H}_{(1;1)}\right|
=M_{1n}-\sum_{k=j}^{n-1} M_{1k} \cdot\left|{H}_{[n-k]}^{(k+1)}\right|
+\\
\sum_{l=2}^{j-1}\sum_{k=j}^{n-1} \left|\left({H}_{(1;1)}\right)_{[l-1]}\right|\cdot M_{lk}\cdot 
\left|{H}_{[n-k]}^{(k+1)}\right|
-
\sum_{l=2}^{j-1} \left|\left({H}_{(1;1)}\right)_{[l-1]}\right|\cdot M_{ln}, \label{H11}
\end{multline}
and
\begin{multline}
 \left|{H}_{(2;1)}\right|
= M_{0n}-\sum_{k=j}^{n-1} M_{0k} \cdot \left|{H}_{[n-k]}^{(k+1)}\right|+
\\
\sum_{l=2}^{j-1}\sum_{k=j}^{n-1}  \left|\left({H}_{(2;1)}\right)_{[l-1]}\right| \cdot M_{lk} \cdot \left|{H}_{[n-k]}^{(k+1)}\right|- \sum_{l=2}^{j-1}  \left|\left({H}_{(2;1)}\right)_{[l-1]}\right| \cdot M_{ln}.\label{H12}
\end{multline}

A combination of \eqref{HH1112}, \eqref{H11} and \eqref{H12} yields that
\begin{align}
&\left|H \right|
=-\sum_{k=j}^{n-1} M_{0k}\cdot \left|{H}_{[n-k]}^{(k+1)}\right|+M_{0n}+
\sum_{k=j}^{n-1} M_{01}\cdot M_{1k} \cdot\left|{H}_{[n-k]}^{(k+1)}\right|+ \nonumber\\
&\sum_{l=2}^{j-1}\sum_{k=j}^{n-1} \left|{H}_{[l-1]} \right|\cdot  M_{lk}\cdot \left|{H}_{[n-k]}^{(k+1)}\right|-
M_{01}\cdot M_{1n} -\sum_{l=2}^{j-1} \left|{H}_{[l-1]} \right|\cdot  M_{ln}.\nonumber
\end{align}

Case II: For the more general case that $M_{kk}\neq I_{\mathcal R}$ for some $k$, by \eqref{QuasiRepreElement} the calculation of $\left|H \right|$ can be reduced to $\left|{H}_{\diamond}\right|$ and the latter can be obtained in Case~I.

\end{proof}

\section{Inverse of block triangular matrices}
\label{Sec:Inverse}

Based on the above section, we are in a position to give an explicit formula for each block of the inverse of a noncommutative block triangular matrix.

\begin{theorem}\label{CorUnit}
Let ${M}$ be an $m\times m$ block triangular matrix with the $n$-by-$n$ block decomposition~\eqref{UpperMatrix}, where each block element $M_{jk}\in \mathcal R^{s_j\times s_k}$ $(m=\sum_{j=1}^{n} s_j)$. If each diagonal block $M_{kk}$ is invertible in the ring $\mathcal R^{s_k\times s_k}$, then ${M}$ is invertible  in the ring $\mathcal R^{m\times m}$. In this case, suppose that ${M}^{-1}=:
(M^{[-]}_{jk})_{j,k=1}^n$, where each $M^{[-]}_{jk}\in \mathcal R^{s_j\times s_k}$.
  Then
\begin{equation}\label{wtildeMjk}
M^{[-]}_{jk}=\begin{cases}
-M^{-1}_{jj}\cdot
\left|{M}^{(j,j+1)}_{[k-j]}\right|\cdot M^{-1}_{kk}, & j= 1,\ldots,k-1,\\
M^{-1}_{kk}, & j=k,\\
0^{j\times k}_{\mathcal R}, & j=k+1,\ldots,n.
\end{cases}
\end{equation}
 \end{theorem}

\begin{proof} We only consider the case that $M_{kk}=I_{\mathcal R^{s_k\times s_k}}$ for $k=1,\ldots,n$;
more general case can be transformed to the above case with the help of the multiplication ${M}\cdot {\rm diag}\left(M^{-1}_{11}, \cdots, M^{-1}_{nn}\right)$ and the equation \eqref{QuasiRepreElement}.

 Let ${M}^{[-]}:=(M^{[-]}_{jk})_{j,k=1}^n$, where $M^{[-]}_{jk}$ is as in \eqref{wtildeMjk}. 
For $j,k=1,\ldots,n$, apparently we have
\begin{equation}\label{MM-}
\mbox{row}_k({M})\cdot \mbox{col}_j({M}^{[-]})=\begin{cases} I_{\mathcal R}, & k=j.\\
0_{\mathcal R}, & k=j+1,\ldots,n.
\end{cases}
\end{equation}
and 
\begin{equation}\label{M-M}
\mbox{row}_k({M}^{[-]})\cdot \mbox{col}_j({M})=\begin{cases} I_{\mathcal R}, & k=j,\\
0_{\mathcal R}, & k=j+1,\ldots,n.
\end{cases}
\end{equation}

Now consider the case that $k<j$. By assuming that
\begin{equation*}
{N}:=\begin{bmatrix}
I_{\mathcal R^{k\times k}} & M_{k,k+1}  & \cdots & M_{k,j-1} & M_{kj}\\
I_{\mathcal R^{k\times k}} & M_{k,k+1}  & \cdots & M_{k,j-1} & M_{kj}\\
 & I_{\mathcal R^{(k+1)\times (k+1)}} & \cdots & M_{k+1,j-1} & M_{k+1,j}\\
 & & \ddots &\vdots & \vdots\\
 & & & I_{\mathcal R^{(j-1)\times (j-1)}}  & M_{j-1,j}
\end{bmatrix}\in \mathcal R^{(\sum_{l=k}^{j} s_l) \times (s_k+\sum_{l=k}^{j-1} s_l)},
\end{equation*}
one derives that
\begin{equation}\label{MM-2}
\mbox{row}_k({M})\cdot \mbox{col}_j({M}^{[-]})
=M_{k,j}-\sum_{t=0}^{j-k-1}  M_{k,k+t} \cdot \left|{M}^{(k+t,k+t+1)}_{[j-k-t]}\right|
 =\left|N \right|=0^{s_k\times s_j}_{\mathcal R},
\end{equation}
where the second equality is due to Proposition \ref{FactorHessenberg} and the last one is due to \cite[Proposition, Subsection 1.3, P. 93]{GR}.
It follows from \eqref{MM-} and \eqref{MM-2} that ${M}\cdot {M}^{[-]}=I_{\mathcal R^{m\times m}}$ and
 ${M}$ is right invertible.

Analogously, a combination of Proposition \ref{FactorHessenberg} and  \cite[Proposition, Subsection 1.3, P. 93]{GR} yields that
\begin{equation}\label{M-M2}
\mbox{row}_k({M}^{[-]})\cdot \mbox{col}_j({M})
=M_{kj}-\sum_{t=1}^{j-k} \left|{M}^{(k,k+1)}_{[t]}\right| M_{k+t,j}=\left|\wtilde {N}\right|=0^{s_k\times s_j}_{\mathcal R},
\end{equation}
where
\begin{equation*}
\wtilde {N}:=\begin{bmatrix}
 M_{k,k+1} & \cdots & M_{k,j-1} & M_{kj} &M_{kj} \\
 I_{\mathcal R^{(k+1)\times (k+1)}} & \cdots & M_{k+1,j-1} &  M_{k+1,j}& M_{k+1,j} \\
  & \ddots &\vdots  & \vdots& \vdots\\
  &&I_{\mathcal R^{(j-1)\times (j-1)}} & M_{j-1,j}  & M_{j-1,j}\\
    & &  & I_{\mathcal R^{j\times j}}  & I_{\mathcal R^{j\times j}}
\end{bmatrix}\in \mathcal R^{^{(s_j+\sum_{l=k+1}^{j} s_l) \times (\sum_{l=k}^{j} s_l)}}.
\end{equation*}
By \eqref{M-M} and \eqref{M-M2} one can see that ${M}^{[-]}\cdot {M}=I_{\mathcal R^{m\times m}}$ and
 ${M}$ is left invertible.

Consequently ${M}$ is invertible and ${M}^{ -1}={M}^{[-]}$.

\end{proof}

On grounds of Theorem \ref{CorUnit}, it becomes transparent  that each $(j,k)$-th block element of ${M}^{-1}$ is uniquely determined by all the $\frac12(k-j+1)(k-j+2)$ blocks of the submatrix 
$$
 \begin{bmatrix}
M_{jj} &  \cdots& M_{jk}\\
&  \ddots&\vdots \\
& & M_{kk}
\end{bmatrix}.
$$
Although the connection between blocks of ${M}$ and its inverse is now established,  a further question might  arise: What is the relationship among blocks of the inverse ${M}^{-1}$? The answer is revealed in the following theorem where one can find a different type of recursive expressions for each block of ${M}^{-1}$.
\begin{theorem}\label{RecursiveInverse}
Under the same assumption as in Theorem \ref{CorUnit} and $j$, $k=1,\ldots,n$ with $j\leq k$,  $M_{jk}^{[-]}$ can be recursively derived via, for any $l=1,\ldots,k-j$,
\begin{equation*}
M_{jk}^{[-]}=-
\sum_{p=0}^{l-1} \sum_{q=l}^{k-j} M_{j,p+j}^{[-]} M_{j+p,j+q} M_{j+q,k}^{[-]}, 
\end{equation*}
or equivalently, in the matrix form
\begin{equation*}
M_{jk}^{[-]}=-\left[M_{jj}^{[-]}, \cdots, M_{j,j+l-1}^{[-]} \right]\cdot M_{[l,k-j-l+1]}^{(j,j+l)} \cdot
\begin{bmatrix}
M_{j+l,k}^{[-]}\\
\vdots\\
M_{kk}^{[-]}
\end{bmatrix}.
\end{equation*}
\end{theorem}

\begin{proof}For $p=0,\ldots,l-1$ and  $q=l,\ldots,k-j$, suppose that ${M}_{p,q,j,k}$ is a $s_j\times s_k$ matrix over $\mathcal R$ given by
\begin{equation*}
{M}_{p,q,j,k}:=M_{j,p+j}^{[-]} M_{j+p,j+q} M_{j+q,k}^{[-]}.
\end{equation*}
According to Theorem \ref{CorUnit}, one deduces that 
\begin{align}
&M_{jj}{M}_{p,q,j,k}M_{kk}= \nonumber\\
&\begin{cases}
 M_{jk}, & p=0, q=k-j, \\
-M_{j,j+q}M_{j+q,j+q}^{-1} \left|{M}_{[k-j-q]}^{(q,q+1)}\right|,  & p=0, q=l,\ldots,k-j-1,\\
-\left|{M}_{[p]}^{(j,j+1)}\right|M_{j+p,j+p}^{-1} M_{j+p,k}, & p=1,\ldots,l-1, q=k-j, \\
\left|{M}_{[p]}^{(j,j+1)}\right|M_{j+p,j+p}^{-1} M_{j+p,j+q} M^{-1}_{j+q,j+q} 
\left|{M}_{[k-j-q]}^{(q,q+1)}\right|,  & \mbox{others}.
\end{cases} \nonumber
\end{align}
Coupled with Proposition \ref{FactorHessenberg}, the above equation reveals that  the summation of all $M_{jj}{M}_{p,q,j,k}M_{kk}$ is equal to a block Hessenberg quasideterminant $\left|{M}^{(j,j+1)}_{[k-j]}\right|$. Thus one can 
 complete the proof by using Theorem \ref{CorUnit} again.

\end{proof}

As an immediate consequence of Theorem \ref{RecursiveInverse}, at the end of this paper we give an off-diagonal  perturbation analysis for the inverse calculation: Let ${M}$ be a  block triangular matrix with the block decomposition~\eqref{UpperMatrix}.
 For any $l=1,\ldots,n-1$, $M$ can be decomposed into the $2$-by-$2$ block form
$$
M=\left[ \begin{array}{c|c}
{M}_{[l]} & {M}_{[l,n-l]}^{(1,l+1)} \\ \hline
0 & {M}_{[n-l]}^{(l+1,l+1)}
\end{array}
\right].
$$
If one imposes a  perturbation matrix  $E$ on the  block submatrix $ {M}_{[l,n-l]}^{(1,l+1)}$ lying in the  off-diagonal rectangular part  of ${M}$, it is required to see how each block element of  ${M}^{-1}$ accordingly changes  under this perturbation.

\begin{corollary}\label{ThmPermutation}
Let ${M}$ be an $m\times m$  block triangular matrix with the $n$-by-$n$ block decomposition~\eqref{UpperMatrix}, where each $M_{jk}\in \mathcal R^{s_j\times s_k}$ $(m=\sum_{j=1}^{n} s_j)$. Given any $l=1,\ldots,n$ and a block matrix $E:=[E_{jk}]_{j,k=1}^{l,n-l}$, where each $E_{jk}\in \mathcal R^{s_{j}\times s_{l+k}}$, suppose that  $\wtilde M$ is an $m\times m$  matrix with the $2$-by-$2$ block form
$$
\wtilde M=\left[ \begin{array}{c|c}
{M}_{[l]} & {M}_{[l,n-l]}^{(1,l+1)}+E \\ \hline
0 & {M}_{[n-l]}^{(l+1,l+1)}
\end{array}
\right].
$$
Assume that ${M}^{-1}=:[M^{[-]}_{jk}]_{j,k=1}^n$ and $\wtilde {M}^{-1}=:[\wtilde M^{[-]}_{jk}]_{j,k=1}^n$, where each  $M^{[-]}_{jk}\in \mathcal R^{s_j\times s_k}$ and $\wtilde M^{[-]}_{jk}\in \mathcal R^{s_j\times s_k}$, respectively. Then
\begin{equation*}
\wtilde M^{[-]}_{jk}-M^{[-]}_{jk}=\begin{cases}
-\left[M_{jj}^{[-]}, \cdots, M_{jl}^{[-]} \right]\cdot E_{[l-j+1,k-l]}^{(j,1)} \cdot
\begin{bmatrix}
M_{l+1,k}^{[-]}\\
\vdots\\
M_{kk}^{[-]}
\end{bmatrix}, & j=1,\ldots,l,\\
&  k=l+1,\ldots,n,\\
0_{\mathcal R}, & \mbox{others}.
\end{cases}
\end{equation*}

\end{corollary}

\section{An example}\label{Sec:Example}
Let $\mathcal Q$ be the skew-field of quaternions as:
$$
\mathcal Q\coloneqq  \left\{a+b \bi +c \bj +d \bk| a,b,c,d \in \mathbb R   \right\}
$$
where the imaginary units $\bi$, $\bj$ and $\bk$  satisfy the rules
$$
\bi^2=\bj^2=\bk^2=-1,\ \bi \bj=-\bj \bi=\bk, \  \bj \bk=-\bk \bj=\bi,\  \bk \bi=-\bi \bk=\bj.
$$
Let $M$ be a block upper triangular $8\times 8$ matrix over $\mathcal Q$ with the block decomposition
\begin{equation*}
 {M}\coloneqq \begin{bmatrix}
 M_{11} & M_{12} & M_{13} & M_{14} & M_{15} \\
  & M_{22} & M_{23} & M_{24} & M_{25}\\
  & & M_{33}  & M_{34} & M_{35}\\
  & &  & M_{44}  & M_{45}\\
  & &  & & M_{55}
 \end{bmatrix},
 \end{equation*}
 where $$
\begin{aligned}
&M_{11}\coloneqq \bi+\bk, \  M_{12}\coloneqq [1+\bi-2\bj+\bk, 2-\bi],\ M_{13}\coloneqq 2-3\bi+4\bk,\\ 
&M_{14}\coloneqq [3+2\bi-5\bj-\bk, 2-\bj, 4], \ M_{15}\coloneqq 2+\bk,  \\
&M_{22}\coloneqq \begin{bmatrix}
\bi & \bj\\
1 & \bk
\end{bmatrix},\  M_{23}\coloneqq \begin{bmatrix}
1+\bk \\
3-\bi+\bj 
\end{bmatrix},\\
&  M_{24}\coloneqq \begin{bmatrix}
\bi-\bj & 5+\bi-\bk & 4-\bi \\
4+2\bj+\bk & 4 & 2-3\bi+\bj+2\bk
\end{bmatrix},\\
&M_{25}\coloneqq \begin{bmatrix}
3+\bj  \\
1+\bj-\bk 
\end{bmatrix},
\ M_{33}\coloneqq 2+\bi-\bk,\\
&M_{34}\coloneqq [-\bj+2\bk, 3\bi-\bj+2\bk, 6-4\bk],\ M_{35}\coloneqq 1+\bi,\\
&  M_{44}\coloneqq \begin{bmatrix}
1 & 0 & \bi\\
1-\bk & \bj & 0\\
\bk & 0 &0
\end{bmatrix}, 
\  M_{45}\coloneqq \begin{bmatrix}
2+\bj\\
1-\bi \\
5\bk 
\end{bmatrix}, 
\  M_{55}\coloneqq 1+2\bi+\bj-\bk.
\end{aligned}
$$

Decomposite the inverse $M^{-1}$ of $M$ into  $M^{-1}=:[M^{[-]}_{jk}]_{j,k=1}^5$,
where each block $M^{[-]}_{jk}$ has the same size as 
$M_{jk}$. The theoretical results in this work can help us to answer the following question:
\begin{itemize}
\item What is the value of $M^{[-]}_{15}$?
\end{itemize}
Moreover,  suppose that by perturbation $M$ becomes
\begin{equation*}
 \wtilde {M}\coloneqq \begin{bmatrix}
 M_{11} & M_{12} & M_{13} & M_{14}+E_{14} & M_{15}+E_{15} \\
  & M_{22} & M_{23} & M_{24}+E_{24} & M_{25}+E_{25}\\
  & & M_{33}  &  M_{34}+E_{34} & M_{35}+E_{35}\\
  & &  & M_{44}  & M_{45}\\
  & &  & & M_{55}
 \end{bmatrix},
 \end{equation*}
 where 
$$
\begin{aligned}
&E_{14}\coloneqq [\bi, \bi, \bi], \ E_{15}\coloneqq \bi,  \\
& E_{24}\coloneqq \begin{bmatrix}
\bj & \bj & \bj \\
\bi+\bj+\bk & \bi+\bj+\bk & \bi+\bj+\bk
\end{bmatrix},\ E_{25}\coloneqq \begin{bmatrix}
\bj  \\
\bi+\bj+\bk 
\end{bmatrix},\\
&   E_{34}\coloneqq [\bk, \bk, \bk],\ E_{35}\coloneqq \bk.
\end{aligned}
$$
Decomposite the inverse $\wtilde M^{[-]}$ of $\wtilde M$ into  $\wtilde M^{-1}=:[\wtilde M^{[-]}_{jk}]_{j,k=1}^5$. A further question is:
\begin{itemize}
\item
What is the absolute value $
\left|\wtilde M^{[-]}_{15}-M^{[-]}_{15}\right|
$
  and the relative value $\frac{
\left|\wtilde M^{[-]}_{15}-M^{[-]}_{15}\right|}{\left|M^{[-]}_{15}\right|}$
?
\end{itemize}

Our answer is given as follows: We first compute the diagonal blocks of $M^{-1}$:
$$
\begin{aligned}
&M_{11}^{[-]}=M_{11}^{-1}=-\frac12 (\bi+\bk),\  M_{22}^{[-]}=M_{22}^{-1}=\frac12\begin{bmatrix}
-\bi & 1\\
-\bj & -\bk
\end{bmatrix}, \\
&M_{33}^{[-]}=M_{33}^{-1}=\frac16 (2-\bi+\bk), \\
&  M_{44}^{[-]}=M_{44}^{-1}=\begin{bmatrix}
0 &  0 & -\bk\\
0 & -\bj & -(\bi+\bj)\\
-\bi & 0 & \bj
\end{bmatrix},
\\
& M_{55}^{[-]}=M_{55}^{-1}=\frac17 (1-2\bi-\bj+\bk).
\end{aligned}
$$
It follows that 
$$
\begin{aligned}
&M_{12}^{[-]}=-M_{11}^{[-]}M_{12}M_{22}^{[-]}=\frac12[1+2\bi-\bk,2\bi+\bj-\bk],\\
&M_{45}^{[-]}=-M_{44}^{[-]}M_{45}M_{55}^{[-]}=\frac17\begin{bmatrix}
-5+10\bi+5\bj-5\bk\\
5+2\bi-11\bj-12\bk\\
-7-2\bi+\bj+4\bk
\end{bmatrix},
\end{aligned}
$$
and then, by Theorem \ref{RecursiveInverse},
$$
M_{13}^{[-]}=-[M_{11}^{[-]},M_{12}^{[-]}] 
\begin{bmatrix}
M_{13}\\
M_{23}
\end{bmatrix}
 M_{33}^{[-]}
=-\frac1{12} (17+19\bi +13\bj+9\bk).$$
Due to Theorem \ref{RecursiveInverse} again,
$$
\begin{aligned}
M_{15}^{[-]}=&-[M_{11}^{[-]},M_{12}^{[-]},M_{13}^{[-]}]\begin{bmatrix}
M_{14} & M_{15}\\
M_{24} & M_{25}\\
M_{34} & M_{35}
\end{bmatrix}
\begin{bmatrix}
M_{45}^{[-]}\\
M_{55}^{[-]}
\end{bmatrix}\\
=&-\frac{1033}{84}+\frac{1051}{84} \bi+\frac{193}{84} \bj+ \frac{701}{28} \bk.
\end{aligned}
$$
Using Corollary \ref{ThmPermutation}, we have
$$
\begin{aligned}
\wtilde M^{[-]}_{15}-M^{[-]}_{15}=&-[M_{11}^{[-]},M_{12}^{[-]},M_{13}^{[-]}]\begin{bmatrix}
E_{14} & E_{15}\\
E_{24} & E_{25}\\
E_{34} & E_{35}
\end{bmatrix}
\begin{bmatrix}
M_{45}^{[-]}\\
M_{55}^{[-]}
\end{bmatrix}\\
=&\frac{10}{21}+\frac{1}{7} \bi- \frac{11}{21} \bj+ \frac{20}{21} \bk.
\end{aligned}
$$
Therefore,
the absolute value $
\left|\wtilde M^{[-]}_{15}-M^{[-]}_{15}\right|
$
  and the relative value $\frac{
\left|\wtilde M^{[-]}_{15}-M^{[-]}_{15}\right|}{\left|M^{[-]}_{15}\right|}
$ are approximately $1.1953$ and $0.0389$, respectively.

\section*{Acknowledgement}
The author is supported by the Foundation for Fostering Research of Young Teachers in South China Normal  University (Grant No. 19KJ20).

\end{document}